\documentclass{amsart}

\usepackage{amsmath,amssymb}
\usepackage{amsthm}

\newcommand{\BO}{\mathsf{BO}}

\newcommand{\rk}{\mathsf{r}}

\newcommand{\ac}{\mathcal{A}}
\newcommand{\Z}{\mathbb{Z}}
\newcommand{\N}{\mathbb{N}}
\newcommand{\Zp}{\mathbb{Z}/p\mathbb{Z}}
\newcommand{\Zn}{\mathbb{Z}/n\mathbb{Z}}

\newcommand{\binomial}[2]{{#1 \choose #2}}

\newtheorem{theorem}{Theorem}[section]
\newtheorem{lemma}[theorem]{Lemma}
\newtheorem{corollary}[theorem]{Corollary}
\newtheorem{proposition}[theorem]{Proposition}
\newtheorem{theirtheorem}{Theorem}

\begin{document}

\title[Barycentric-sum problems over cyclic groups]{Some remarks\\
on barycentric-sum problems over cyclic groups}
\author{O.~Ordaz \and A.~Plagne \and W.~A.~Schmid}

\address{Escuela de Matem\' aticas y Laboratorio MoST, Centro ISYS, Facultad de Ciencias,
Universidad Central de Venezuela, Ap. 47567, Caracas 1041--A, Venezuela}
\address{CMLS, \'Ecole polytechnique, 91128 Palaiseau Cedex, France}
\address{Universit\'e Paris 13, Sorbonne Paris Cit\'e, LAGA, CNRS, UMR 7539, Universit\'e Paris 8, F-93430, Villetaneuse, France}
\address{}

\email{oscarordaz55@gmail.com}
\email{plagne@math.polytechnique.fr}
\email{schmid@math.univ-paris13.fr}

\thanks{The research of O. Ordaz is supported by Banco Central de Venezuela and Postgrado de la Facultad de Ciencias de la U.C.V.; the one of A. Plagne by the PHC Amadeus 2012: projet n\textsuperscript{o}~27155TH;
the one of W. A. Schmid by the Austrian Science Fund (FWF): J 2907-N18, and the PHC Amadeus 2012: projet n\textsuperscript{o}~27155TH}

\dedicatory{To the memory of Yahya Ould Hamidoune}.

\begin{abstract}
We derive some new results on the $k$-th barycentric Olson constants of abelian groups (mainly cyclic). This quantity, for
a finite abelian (additive) group $(G,+)$, is defined as the smallest integer $\ell$ such that each subset $A$ of
$G$ with at least $\ell$ elements contains a subset with $k$ elements $\{g_1, \dots, g_k\}$ satisfying
$g_1 + \cdots + g_k = k\ g_j$ for some $1 \leq j \leq k$.
\end{abstract}

\maketitle

\section{Introduction}

It is a classical problem of Additive Combinatorics to find size conditions on subsets of
(or sequences over) additive abelian (semi-)groups which guarantee that the subsets
contain a certain `arithmetic pattern';
two classical `patterns' are arithmetic progressions (of a certain fixed length) and
zero-sum sets or sequences (of a certain cardinality or length).
Since we only consider finite structures, our size condition is always the cardinality
of the set or the length of the sequence.

Motivated among others by results of Y. Ould~Hamidoune, such as \cite{hamdisc}, systematic investigations related to the following `pattern' have
been started in \cite{MT,chabela} and continued for instance in \cite{luong, oscar}; for a survey see \cite{domingo1}.
A very closely related problem also appears in some work of Alon \cite{alon}, and special cases are classical (cf.~below).

We fix a positive integer $k$; to exclude trivial corner cases one often imposes $k \ge 3$.
A set $\{g_1, \dots, g_k\}$, assuming the $g_i$'s are distinct, or a sequence $(g_1, \dots, g_k)$
in a finite abelian (additive) group $(G,+)$ is called \emph{barycentric} if
\[
\sum_{i=1}^k g_i = k\ g_j
\]
for some $1\le j \le k$; this equality is also used in the sometimes more convenient form
\[
\sum_{i=1,\ i \neq j}^k g_i = (k-1)\ g_j.
\]
For brevity, we refer to barycentric sets (resp. sequences) of size $k$ as {\em $k$-barycentric}.

In the present paper, we obtain several results, mainly for $G$ prime cyclic,
related to the problem of determining, for a given value of $k$, a minimal size
above which any set (resp. sequence) in $G$ has to contain a $k$-barycentric subset (resp. subsequence).

We point out that a set is $3$-barycentric if and only if its elements form an arithmetic progression;
for sequences, this is also true
(if one does not insist that the elements of an arithmetic progression are distinct,
in other words if one allows the difference to be an element of order strictly less than the number of terms).
And, for $n$ being (a multiple of) the exponent of the group, being $n$-barycentric
is equivalent to being a zero-sum of size $n$.
So, the above questions contain the problem of guaranteeing the existence
of  $3$-term arithmetic progressions --- its finite abelian group version --- as well as the Erd{\H o}s--Ginzburg--Ziv problem and some of its variants \cite{gaogeroldingersurvey}.
Thus, hoping for explicit and complete answers seems by far too optimistic.
Indeed, the above questions, for values of $k$ between $3$ and the exponent of $G$,
can be seen as an interpolation between these two classical problems.
An aim of this paper is to highlight a certain `phase-transition' that the problem undergoes when, for some fixed cyclic group $G$, the parameter $k$ is varied.
In this article, we focus on the version of the question for sets.

The constants we study are the so-called, for $k$ a positive integer,
$k$-th \emph{barycentric Olson constants} of a finite abelian group $G$, denoted by  $\BO(k, G)$ and
defined as the smallest integer $\ell$ such that each set $A$ with at least $\ell$ elements
contains a $k$-barycentric subset. We note that this is always well-defined as for $\ell> |G|$ the condition becomes vacuously true; in other words we have $\BO(k, G)\le |G|+1$. Another possibility is to only define the constant if a $k$-barycentric set exists at all; indeed, that convention is used in various earlier work on the problem. Yet, by analogy with work on other constants, see e.g. \cite{gaogeroldingersurvey}, we use the former convention.

The name is derived from that of the Olson constant, the present constant is the barycentric analog of that zero-sum constant (see, e.g., \cite{gaogeroldingersurvey}).

\section{Notation and preliminaries}

In this article, the main structures we study are cyclic (and sometimes other finite) abelian groups, denoted additively. In some of our arguments it will be convenient to consider $\Zn$, the integers modulo $n$ instead of an abstract cyclic group of order $n$, in order to have a natural notion of multiplication available as well. Typically, we do not make any notational distinction between integers and their residue classes; in the present work there seems little risk for confusion so that this notational simplification seems justified.

Let $(G,+)$ be an abelian group.
The (Minkowski) sumset of two subsets $A$ and $B$ of $G$
will be denoted by
\[
A+B = \{ a+b \colon a \in A, b\in B    \} .
\]
We denote the sum of the elements of a subset $S$ of $G$ by $\sigma (S)$.
Furthermore, for $k$ an integer, let
\[
\Sigma_k (A) =\{ \sigma (S) \colon S \subset A,\, |S|=k \}.
\]
Finally, for $t$ an integer, we denote by $t \cdot A$  the set of multiples $\{ ta \colon a \in A  \}$.

Recall that each finite abelian group $G$ is isomorphic to a product
\[
\Z  /n_1 \Z \times \Z  /n_2 \Z  \times \cdots \times  \Z  /n_r \Z,
\]
with integers $1< n_1 | n_2 | \cdots | n_r$ that are uniquely determined. The integer $r$ is called the rank of $G$, denoted $\rk(G)$. For a prime
$p$, the $p$-rank of $G$, denoted $\rk_p (G)$, is defined as the number of $i$ such that  $n_i$ is divisible by $p$. Moreover, $n_r$ is called the exponent of $G$, denoted $\exp(G)$.

Recall that the multiplicative order of $s$, an integer coprime to $n$, in $\Zn$ is the smallest positive integer $t$
such that $s^t \equiv \; 1 \pmod{n}$.

\section{Some general results}

While our main object of study is $\BO(k, G)$ for $G$ a cyclic group (of prime order),
we start by summarizing and establishing some results valid for general finite abelian groups. In particular we treat various corner cases; to be able to exclude them is convenient in our consideration for cyclic groups as well.
In special cases, mainly cyclic groups, some of them were known, and others are more or less direct and included for definiteness only, yet in particular the case of $k=|G|-2$, in general, is new and seems to be of some interest in its own right.

Let $(G,+)$ be a finite abelian group.
The first observations are direct consequences of the definition.
For every $k$ we have that $\BO(k,G) \le |G|+1$, as the conditions becomes trivially true; for $k >|G|$ we of course have $\BO(k,G)= |G|+1$, since the group $G$ cannot contain a $k$-barycentric subset (this is however not the only case where this happens, compare below). Conversely, for $k \le |G|$ we have $\BO(k,G) \ge k$.

Since any singleton is $1$-barycentric and since no $2$-barycentric set can exist, we have
\[ \BO(1,G)=1 \text{ and } \BO(2,G)=|G|+1.\]
Yet, as mentioned in the introduction the case $k=3$ is, by contrast, a very challenging problem. We now turn to large values of $k$.

For $k=|G|$, one has to decide whether the full group is $|G|$-barycentric or not, that is whether $\sigma(G)=0$.
It can be seen that only in case when $\rk_2(G)=1 $, the sum $\sigma(G)$ is non-zero.
More precisely (cf.~\cite{gaogerol2} for a detailed argument) the following is true.

\begin{lemma}
\label{sigmaG}
Let $G$ be a finite abelian group. We have
\[
\sigma (G) =
\begin{cases}
b, & \text{ if } \rk_2 (G) =1 \text{ and $b$ denotes the unique element of order $2$},\\
0, & \text{ otherwise.}
\end{cases}
\]
\end{lemma}

As a consequence, one gets that
\[
\BO (|G|,G) =\begin{cases}
|G|+1, &  \text{ for } \rk_2(G)=1 ,\\
|G|,  &  \text{ otherwise.}
\end{cases}
\]

We may also compute $\BO (k,G)$ in the cases $k=|G|-1$ and $|G|-2$.  The values of these constants are given by the following proposition.

\begin{proposition} Let $G$ be a finite abelian group. We have, for  $|G|\ge 2$,
\[
\BO (|G|-1,G) =
\begin{cases}
|G|-1, & \text{ if } \rk_2 (G) =1,\\
|G|+1, & \text{ otherwise,}
\end{cases}
\]
and, for  $|G|\ge 3$,
\[
\BO (|G|-2,G) =
\begin{cases}
|G| - 2, & \text{ if } |G| \text{ is odd,}  \\
|G| + 1, & \text{ if  $\exp(G)= 2$ or $|G|=4$,}\\
|G| - 1, & \text{ otherwise.}
\end{cases}
\]
\end{proposition}

\begin{proof}
Let us first consider the case $k=|G|-1$. Let $A=G \setminus \{ a \}$ be a set of cardinality $|G|-1$. The only $b$ in $G$ for which
$\sigma (A) =(|G|-1)b$ is $-\sigma (A) = a -\sigma (G)$, which is, by Lemma \ref{sigmaG}, different from $a$ (and thus an element of $A$)
if and only if the $2$-rank of $G$ is equal to $1$ (and thus $\sigma(G)\neq 0$).
That is, if the $2$-rank of $G$ is $1$ every set of cardinality $|G|-1$ is barycentric while
otherwise no $(|G|-1)$-barycentric set exists.

We now consider the case $k=|G|-2$; recall $|G|\ge 3$. Let $A=G \setminus \{ a, a' \}$ be a set of cardinality $|G|-2$
(that is, $a \neq a'$). We distinguish two main cases, according to whether $\rk_2(G)=1$ or not.

For $\rk_2(G)$ not equal to $1$, again by Lemma \ref{sigmaG}, we get that $\sigma(A)= - a - a'$.
Since $(|G|-2) b = -2b$ for each $b \in G$, we get that if $\sigma(A) =k b$ for some $b \in G$, then $b \notin \{a,a'\}$.
This implies that $A$ is barycentric if and only if $\sigma (A) \in k \cdot G = -2 \cdot G= 2 \cdot G$.

Thus $\BO (|G|-2,G)= |G|-2$ for $\rk_2 (G)=0$, since in this case $2 \cdot G =G$.
In case $\rk_2 (G)>1$, we observe that $\BO (|G|-2,G)> |G|-2$, as for suitable $a,a'$ we can make it so that $a+a' \notin 2\cdot G$;
indeed, note that $\Sigma_2(G)=G$ for $\exp(G)> 2$ and $\Sigma_2(G)=G\setminus \{0\}$ for $\exp(G)=2$. Moreover, this last observation already implies that for $\exp(G)=2$ one has
 $\BO (|G|-2,G)= |G|+1$, as $2\cdot G =\{0\}$ and therefore no $(|G|-2)$-barycentric set can exist.

We suppose the exponent of $G$ is not $2$, that is $|2\cdot G|>1$ and consider a subset $B=G \setminus \{a^{\ast}\}$ of cardinality $|G|-1$. Due to the translation invariance of the problem we may assume $a^{\ast}=0$.
In order to show that $\BO (|G|-2,G)= |G|-1$ it suffices to find some subset $A$ of $B$ that is $k$-barycentric.
This is equivalent to finding an $A$ of cardinality $k$ whose sum is in $2\cdot G$.
Yet for $A = B \setminus \{c\}$ we have  $\sigma(A) = - c$ and it thus suffices to choose some $c \in 2\cdot G \setminus \{0\}$, which is non-empty by assumption.

For $\rk_2(G)$ equal to $1$, Lemma \ref{sigmaG} implies that $\sigma(A)= b - a - a'$
where $b$ is the element of $G$ of order $2$.
Again, for $A$ to be barycentric, a necessary condition is that $\sigma(A) \in 2 \cdot G$. Since in this case
$2 \cdot G$ is not the full group, we get that $\BO (|G|-2,G)> |G|-2$ as each element of $G$ can be written as
$b - a - a'$ for suitable distinct $a$ and $a'$.
Yet in this case the condition $\sigma(A) \in 2 \cdot G$ might not be sufficient, since
$b-a-a' = -2c$ does not necessarily imply that $c$ is neither $a$ nor $a'$.  Indeed, we observe that
in case $b + a = a' $ we have $b-a-a' = -2a = -2a'$, and $a$ and $a'$ are the only solutions of the equation
$b-a-a' = -2c$ for $c \in G$.

Similarly as above we consider a subset $B=G \setminus \{a^{\ast}\}$ of cardinality $|G|-1$, and again assume $a^{\ast}=0$. We seek to find a $c\in B$ such that $A= B \setminus \{c\}$ is barycentric.
For this to happen we need on the one hand that $\sigma(A) = b - c \in 2 \cdot G$ and that $c$ is not the element of order $2$. That is we need that  $b +  (2 \cdot G \setminus \{0\}) \cap B$ is non-empty, that is
$(b +  (2 \cdot G \setminus \{0\})) \setminus \{0\}$ is non-empty. This is the case if  $|2\cdot G|>2$.
Since $\rk_2(G) = 1$, we have $|2\cdot G|>2$ except for $G$ being a cyclic group of order $4$.
Indeed, for a cyclic group of order $4$, the result is different, and it follows by the observation for $k=2$, mentioned before.
\end{proof}

From now on, we shall therefore assume that
\[
3 \leq k \leq |G|-3,
\]
and thus we can also assume that $|G|\ge 6$.
In the following, we are mainly interested in the case $G=\Z/p\Z$ (with $p$ a prime), and can restrict to considering primes $p\ge 7$. We impose these restrictions consistently, even if in some results they are not needed.

\section{Results on cyclic groups}

We now consider specifically the case of cyclic groups starting from the following result which is contained in \cite{chabela} (see Theorem 4 there).

\begin{theirtheorem}
\label{W1}
Let $p\ge 7$ be a prime and $k$ be an integer such that $3 \leq k \leq p-3$. The following bounds hold
\[
k \leq \BO(k, \Zp)\le\left\lceil \frac{p-1}{k-1}\right\rceil + k.
\]
\end{theirtheorem}

The Theorem of Dias da Silva--Hamidoune \cite{DH}, which we recall here as Theorem~\ref{DSH}, is a crucial tool in its proof.

\begin{theirtheorem}\label{DSH}
Let $p$ be a prime and let $A$ be a subset of $\Zp$. For $k$ a positive integer, one has
\[|\Sigma_k(A)|  \ge \min \{ p, k ( |A| - k  ) + 1 \}.\]
\end{theirtheorem}

For the convenience of the reader and since it of some relevance for the subsequent discussion, we include a proof of the upper bound in Theorem \ref{W1} (the lower bound being trivial).

\begin{proof}[Proof of Theorem \ref{W1}]
Notice first that the upper bound satisfies always
\[
\left\lceil \frac{p-1}{k-1}\right\rceil + k \leq p.
\]
We may therefore consider a set $ A \subset \Zp $ with $|A|=\left\lceil \frac{p-1}{k-1}\right\rceil + k$ and
choose an $a \in A$. By Theorem \ref{DSH},
we have
\begin{eqnarray*}
|\Sigma_{k-1} (A\setminus \{ a \} )| & \geq & \min\{p, (k-1)(|A \setminus \{ a \} | -k +1)+1\}\\
										& = & \min\{p, (k-1)(|A| -k ) +1\}= p.
\end{eqnarray*}
Hence $\Sigma_{k-1} (A\setminus \{a\})=\Z /p\Z$ which implies that there
exists $U\subset A\setminus\{a\}$ with $|U|=k-1$ such that $\sigma (U)= a(k-1)$.
Consequently $U \cup \{a\}$ is a $k$-barycentric subset of $A$.
\end{proof}

The quality of the bounds in the above results depends significantly on the relative size of $k$ and $p$. We discuss this in detail below, mentioning the respective known and new results.

For large values of $k$, say if $k\ge cp$ for $c$ a fixed real number ($0<c<1$), the preceding result shows that
\[
k \leq \BO(k,\Zp) \le k + \frac{1}{c} + O(p^{-1})
\]
and the value of $\BO(k,\Zp)$ is known asymptotically up to an additive constant.

More precisely, in the range $(p+1)/2 \le k \le p-3$, Theorem \ref{W1} gives
\[
k \leq \BO(k,\Zp)\le k+2.
\]
We shall improve this and, indeed, we obtain the precise value for $\BO(k,\Zp)$ for $k$ fulfilling this condition. In fact, we prove the following somewhat more general technical result.

\begin{proposition}
\label{proptech}
Let $n\ge 6$ be an integer and $k$ be an integer such that $(n+1)/2 \le k \le n-3$.
Assuming that $n$ is coprime to both $k$ and $k+1$, one has $\BO(k,\Z /n\Z)=k+1.$
\end{proposition}

Of course this result yields the following result as a special case, establishing what we announced just above.

\begin{theorem}
\label{thk+1}
Let $p\ge 7$ be a prime and $k$ be an integer such that $(p+1)/2 \le k \le p-3$, then $\BO(k, \Zp ) = k + 1$.
\end{theorem}

In a certain sense $k = (p+1) / 2$ is an actual threshold.
We discuss in detail the value just below, that is the case $k = (p-1)/2$, where Theorem \ref{W1} gives
\[
k \leq \BO(k,\Zp)\le k+3.
\]
Here we see that the value of $\BO(k,\Zp)$ starts to depend on `number theoretic' properties of $p$.
And, the proof indicates that this is not an isolated phenomenon but related phenomena are to be expected for other $k$ somewhat below $p/2$.

The precise result we obtain is the following.

\begin{theorem}
\label{th8}
Let $p \ge 7$ be a prime and $k = (p-1)/2$. Then,
\[
\mathsf{BO}(k,\mathbb{Z}/p\Z )=
\begin{cases}
k+1, & \text{ if the multiplicative order of 2 modulo $p$ is odd},\\
k+2, & \text{ if it is even}.
\end{cases}
\]
\end{theorem}

Just the bounds in this theorem can be extended to further values of $k$.

\begin{theorem}
\label{UBpol}
Let $p \ge 7$ be a prime and $k$ be an integer such that $ (p+2)/3 \le k \le p-3$.
One has
\[
k+1 \le \mathsf{BO}(k,\Z/ p\Z) \le k+2.
\]
\end{theorem}

The lower bound is obtained by a direct construction (see Lemma \ref{lbk+1}), the upper bound is obtained via an application of the polynomial method in its classical form by Alon, Nathanson, and Ruzsa \cite{ANR}.
It would not be difficult to derive some (incomplete) criteria when equality at the upper or lower bound holds; we do however not address this problem. By contrast, giving a full and meaningful characterization seems difficult.

For ever smaller values of $k$, the problem seems to become increasingly difficult. We however note that, as a consequence of Theorem \ref{W1}, we get that if $k/p^{1/2}$ tends to infinity then
\[\BO(k, \Zp) \sim k\]  as $p$ (and $k$) tends to infinity.
And, if we only assume that $k / p^{1/2}$ is bounded away from $0$,
then $\BO(k, \Zp) /k$ remains bounded and therefore the order of $\BO(k, \Zp)$ is still $k$.

It is not clear to us what the precise or, more realistically, only more precise values should be; we thus refrain from detailed speculations but still include some remarks. It seems likely that for $k$ in this range the upper bound in Theorem \ref{W1} is not that close to the true value. On the one hand, before one applies Theorem \ref{DSH}, the Theorem of Dias da Silva--Hamidoune, there is already some loss (roughly speaking, avoiding this loss is what we do to obtain the upper bound in Theorem \ref{UBpol}).
On the other hand, it is generally believed that near extremal cases in the Theorem of Dias da Silva--Hamidoune arise from sets being close to arithmetic progressions; for such results for the case $k=2$ see, e.g., \cite{karolyi,lev}. And, a set being close to an arithmetic progression should contain large barycentric subsets; see just below for some details.

It was noted in \cite{luca} that, for $k$ fixed, for cyclic groups of sufficiently large order, upper bounds for $\BO(k,G)$ can be obtained from Szemer\'edi's theorem and its quantitative refinements. To avoid potential confusion we add the following explanation. While Szemer\'edi's theorem  and its refinements are typically stated as results on subsets of integers (as opposed to subsets of cyclic groups), one can of course use them to derive analog results on subsets of cyclic groups. In fact this would be sufficient for our purpose, yet, to put this in context we mention in addition that under suitable assumptions, also the converse is true, that is a result established for cyclic groups implies one for subsets of the integers;
in fact, to pass to cyclic groups, to establish a result for cyclic groups, and to reinterpret it for integers is the common general framework for (recent) results around Szemer\'edi's theorem.  So, conversely, one cannot expect (substantially) better results for cyclic groups than those known for the integers.

We mention how Szemer\'edi's theorem is applicable to this problem. It can be shown that an arithmetic progression of odd length is barycentric and an arithmetic progression of even length contains a barycentric subset of cardinality smaller by two. Thus, to obtain an upper bound for the $k$-th barycentric Olson constant of a cyclic group of large order, one can invoke Szemer\'edi's theorem or a refined version thereof, to argue that a set of sufficiently large cardinality contains an arithmetic progression of length $k$ or $k+2$, depending on whether $k$ is odd or even, respectively; this then implies the existence of a $k$-barycentric subset of the set.
In particular, Szemer\'edi's theorem implies that for fixed $k$ the linear (in the order of the group) upper bound provided by Theorem \ref{W1} does not yield the correct order anymore (see \cite{luca}).

As mentioned in the introduction, for $k=3$ the two problems, barycentric Olson constant and quantitative version of Szemer\'edi's theorem, are particularly closely related as for three elements (the case $k=3$) being barycentric and being in arithmetic progression are equivalent.
Specifically, we recall that a recent result of Sanders \cite{sanders} asserts (using the equivalence mentioned above), for some $c>0$,
\[  \BO (3, \Z /n\Z) \leq c n \frac{(\log \log n)^5}{\log n} . \]

However, for other $k$,  being barycentric and being in arithmetic progression are not equivalent; for example as indicated above,
for $k\ge 5$ odd, being barycentric is a (strictly) weaker propery than being in arithmetic progression.
So, one can hope to get upper bounds for $\BO (k, \Z /n \Z)$ that are better than those following from results establishing the existence of arithmetic progressions.
Indeed, for $k\ge 6$ fixed and $n$ sufficiently large, results of Schoen and Shkredov \cite{SS} yield
an upper bound of the form
\[
\BO (k, \Z /n \Z) \leq c' n \exp \left( - c \left( \frac{\log n}{\log \log n} \right)^{1/6}\right),
\]
with $c,c'>0$. To avoid a potential confusion, we point out that we need at least $k=6$;
the closely related example given in the quoted paper \cite{SS} that could suggest that $k=5$
is sufficient is not exactly what we need, as there one has an additional/different variable in the equation.

Based on the classical construction of Behrend \cite{Behrend} a lower bound of this general form can be obtained.
More concretely, Alon \cite{alon} showed
that for $k\ge 2$ there exists a subset of the integers $\{1, \dots, n\}$ of size at least
\[n \exp(-10 \sqrt{ \log k \log n} )\]
such that the equation (in the integers)
\[x_1 + \dots + x_k = k x_{k+1}\]
has no solution in this set except for the trivial ones where all $x_i$'s are equal.

For fixed $k$ and large $n$ this could be translated quite directly into a lower bound for $\BO(k, \Zn)$, where we need to avoid solutions not only in the integers but modulo $n$.
However, we also wish to consider the case that $k$ grows with $n$ (as opposed to $k$ being fixed).
Thus, we carry out a more detailed analysis focusing on this aspect not so common in the various investigations of Behrend-like constructions.

We shall first prove the following technical result.

\begin{theorem}
\label{theoBehrend}
Let $n\ge 6$ and $k$ be two integers such that $3 \le k \le n-3$. One has
\[
\BO(k,\Z/ n \Z) \geq \max_{m \in \N} \left\{ \frac{1}{m} \left( \frac{(n/k)^{1/m} - k}{k-1} \right)^{m-2} \right\}.
\]
\end{theorem}

The following corollary is of special interest since it shows that Behrend's approach
is valid as long as $\log k / \log n$ tends towards $0$ when $n$ goes to infinity.

\begin{corollary}
\label{Behrendkfixed}
When $n$ tends to infinity, if $k \geq 3$ satisfies $\log k = o ( \log n )$, then
\[
\BO(k, \Z/ n \Z) \geq  n \exp (- 5  \sqrt{ \log k \log n}).
\]
\end{corollary}

\section{Proofs}

We start with a direct constructive lower bound, needed in some of our arguments.
Indeed as we will see this bound is actually sharp in not too few cases.

\begin{lemma}
\label{lbk+1}
Let $n\ge 6$ and $k$ be two coprime integers such that $3 \le k \le n-3$. Then $\BO(k,\Zn)\ge k+1$.
\end{lemma}
\begin{proof}
It suffices to give an example of a subset of $\Zn$ of cardinality $k$ that is not barycentric.
We distinguish two cases based on the parity of $k$.
Suppose $k$ is even. We consider the set $A=\{0, \dots, k-1\}$.
Then, $\sigma(A)= k(k-1)/2$. Multiplication by $k$ being an isomorphism, it suffices to decide whether $(k-1)/2$ is in $A$
(note that we know that $n$ is odd). This is not the case as the smallest nonnegative representative of $(k-1)/2$ is $(k-1 + n)/2$, which is greater than $k-1$.

Suppose $k$ is odd. For $k=3$, the set $\{0,1,3\}$ is not barycentric and we assume $k\ge 5$.
We consider the set
\[
A= \left\{ 0, \cdots, \frac{k -1}{2} , \frac{k + 5}{2} , \cdots , k, k + 2 \right\}.
\]
We get $\sigma(A)= k(k+1)/2$, and as above, since $(k+1)/2 \notin A$ the claim follows.
\end{proof}

\subsection{Proof of Theorem \ref{thk+1}}

We actually prove the technical Proposition \ref{proptech}.
By Lemma \ref{lbk+1} it remains to establish the upper bound. Let $A$ be any subset of $\Zn$ with  $k +1$ elements and let $x$ be the sum of all elements of $A$. We have $\Sigma_k (A) = x - A$, as in each sum of $k$ distinct
elements of $A$, exactly one element is missing. It follows that $| \Sigma_k  (A) | =| A | = |k \cdot A |=k+1$, where we use that $k$ and $n$ are coprime.

The intersection of $\Sigma_k (A) $ and $k \cdot A$ contains at least  $2 ( k +1) - n \geq 3$ elements, where we used the lower bound on $k$.
Let $b$ be an element in the intersection different from the element $( k + 1)^{-1} kx$; since $k+1$ and $n$ are coprime the latter is a unique well-defined element and  the intersection contains more than one element.
By definition, there exist $a,a' \in A$ such that $b = ka$ and $b= x - a'$ . Since $b \neq ( k + 1)^{-1} kx$, we have $a\neq a'$.
This implies that  $ka \in k \cdot (A  \setminus \{ a'\} )$ is equal to the sum of the elements
of the set with $k$ elements $A \setminus \{ a'\}$, which is therefore a $k$-barycentric set. This establishes the upper bound.

\subsection{Proof of Theorem \ref{th8}}

In part the general strategy of the proof we give is similar to parts of the proof of Theorem \ref{thk+1}. The latter part, establishing the upper bound $k+2$ is essentially \cite[Theorem 8]{chabela}; it is possible to replace this part by invoking the more general result Theorem \ref{UBpol}, obtained by quite different means, yet we avoid to do this to illustrate the merits of the different methods. Indeed, we did obtain an argument along the lines of \cite[Theorem 8]{chabela} for establishing the upper bound of $k+2$ for $k=(p-3)/2$ as well, yet the details become quite tedious so that we feel that to push this method too far further is infeasible, which necessitates an other approach, whence we subsequently establish Theorem \ref{UBpol}.
In the present case only parts of the argument would, under similar technical conditions as in Theorem \ref{thk+1}, carry over to cyclic groups of composite order, so that we do not make this explicit.

Let $A \subset \Zp$ with $|A|=k+1$. We investigate under which conditions $A$ has a $k$-barycentric subset.

We may assume that the sum of the elements of $A$
is $0$; to see this recall that the problem is invariant under translation. We note that $\Sigma_k (A)= -A$. As above, an element of $-A\cap (k \cdot A)$ yields a $k$-barycentric
subset except if it is $0$. Thus, it follows that if $A$ has no $k$-barycentric subset, then
$-A\setminus \{0\}$ and $k \cdot A \setminus \{0\}$ do not intersect. Since the sum of their cardinalities is
$2|A| -2=2k=p-1$, it follows that $( \Zp ) \setminus \{0\}$ is the disjoint union of $-A\setminus \{0\}$ and
$k \cdot A \setminus \{0\}$. Moreover, note that the converse is true as well (for $A$  having the sum of its
elements equal to $0$).
Multiplying by $-2$ it follows that
$( \Zp ) \setminus \{0\}$ is the disjoint union of
$A\setminus \{0\}$ and $2 \cdot A\setminus \{0\}$.

Let $g \in A\setminus \{0\}$. Then $2g \in 2 \cdot A\setminus \{0\}$.
Thus $2g \notin  A\setminus \{0\}$ and so $4g \notin 2 \cdot A\setminus \{0\}$.
Hence $4g \in A\setminus \{0\}$.
Repeating this argument, we see that $2^ig \in A\setminus \{0\}$ if and only if $i$ is even.

Let $\ell$ denote the multiplicative order of $2$ modulo $p$.
By the above reasoning it follows that if $\ell$ is odd, then $2^{\ell}g \notin  A\setminus \{0\}$.
Yet, $2^{\ell}g = g \in A \setminus \{0\}$, a contradiction.
Thus, if $\ell$ is odd then $A$ contains a $k$-barycentric subset and hence
in this case $\mathsf{BO}(k,\Zp)\le k+1$. Since by Lemma \ref{lbk+1} we know $\mathsf{BO}(k,\Zp) \ge k + 1$ our claim follows in the case of odd $\ell$.

From now on, we assume that $\ell$ is even.
We now use the multiplicative structure of $\Zp$ as well;
of course this does not make the argument inapplicable for an abstract cyclic group of order $p$, one merely would have to define (non-canonically) some multiplicative structure.

Let $H$ be the multiplicative group generated by $2$ and
$\{h_1, \dots, h_m\}$ a set of representatives of the classes---with respect to the multiplicative structure---$((\Zp ) \setminus \{0\})/H$. The set
\[
B=  \bigcup_i \{ 2^{2j}h_i \colon 0 \le j \le \ell/2 - 1 \}
\]
has $k$ elements.
And  $( \Zp )\setminus \{0\}$ is the disjoint union of $B $ and $2\cdot B$; also note that $\sigma(B)=0$.
By the above reasoning it follows that $\{0\} \cup B$ has no $k$-barycentric subset and thus
$\mathsf{BO}(k,\Zp)> k+1$.

It remains to show that $\mathsf{BO}(k,\Zp)\le  k+2$; the argument is essentially \cite[Theorem 8]{chabela}.
Let $C \subset \Zp$ a set with $k+2$ elements.
We may assume that $\sigma(C)=0$.
We note that $\Sigma_k(C)= - \Sigma_2(C) = \Zp$, as $2|C| - 3 =p$.
For each $c \in C$, we thus get $kc= -u-v$ with distinct $u,v\in C$.
This yields a $k$-barycentric sequence except if $c= u$ or $c=v$.
Yet if $kc=-c - v $, then $v = -(k+1)c$.
Thus, either we get a $k$-barycentric set, or for each $c \in C$ it follows that $-(k+1)c \in C$.

It follows that $|C|$, or $|C| - 1$ (in case $0 \in C$), is divisible by the multiplicative order of
$-(k+1)$ modulo $p$, which we denote by $t$ and which is a divisor of $p-1$.

First, suppose $t$ divides $|C|= k+2$. Since $2(k+2) - 4 = p-1$, it follows that $t \mid 4$.
Note that $-(k+1)$ is $-(p+1)/2$ and thus $(-2)^{-1}$ modulo $p$. A contradiction to $p \ge 7$.

Second, suppose $t$ divides $|C|-1= k+1$.
Since $2(k+1) - 2 = p-1$, it follows that $t \mid 2$.  Again, a contradiction to $p\ge 7$.

\subsection{Proof of Theorem \ref{UBpol}}

Since for $p=7$ the result does not yield anything not covered by earlier results, we assume $p\ge 11$.
The lower bound is again a consequence of Lemma \ref{lbk+1}.

We now prove the upper bound. Let $A$ be any subset of $\Z / p\Z$ with cardinality $k+2$.
As before, using a translation, we may assume that $\sigma ( A) = 0$.
We have to show that there is a subset $B$ of $A$ with
cardinality $k$ and an element $b$ in $B$ such that $\sigma (B) = k b$.
Using $| A \setminus B | =2$ and the assumption $\sigma ( A) = 0$,
we observe that what we have to prove is that there are three distinct elements $a,a'$ and $b$ such that
\[
-a -a' =- \sigma (A \setminus B) =\sigma (B) = k b
\]
or, equivalently,
\begin{equation}
\label{aresoudre}
a +a' + k b =0.
\end{equation}
This, in fact, will follow from the fact that if we denote
\[
S_k (A)= \{ a_1+a_2 + k a_3 \colon  a_1, a_2, a_3 \in A \text{ pairwise distinct}\}
\]
the assumption $|A|= k+2 \ge (p+8)/3$ implies $|S_k (A)|=p$, thus $0 \in S$ and there
is at least one solution to \eqref{aresoudre}.

We are now reduced to prove this fact. We will in fact prove slightly more by showing that
the following lemma holds.

\begin{lemma}
Let $p$ be a prime and $k$ be an integer, $3 \leq k \leq p-1$, and $A$ be a subset of $\Z / p\Z$
with cardinality at least $k+2$. We have the following estimates.
\begin{enumerate}
\item[(i)]
For $|A| \leq (p+6)/3$, one has
\[
|S_k (A) | \geq  3|A|-6.
\]
\item[(ii)] For $|A|=(p+7)/3$, one has
\[
|S_k A) | \geq  p-2.
\]
\item[(iii)]
For $|A| \geq (p+8)/3$, one has
\[
|S_k (A)| =p.
\]
\end{enumerate}
\end{lemma}

\begin{proof}
The assertion being trivial otherwise, we may assume $|A|\ge 2$.
We use the polynomial method \cite{ANR} and change slightly our notation to ease
the application of the main result of \cite{ANR}. Let
\[
A_1 \subset A,\ A_2 \subset A,\ A_3 \subset k \cdot A
\]
be non-empty sets and write $c_1 = |A_1 | -1$,  $c_2 = |A_2 | -1$
and  $c_3 = |A_3 | -1$. Since $| k \cdot A | = |A|$, the three $c_i$'s
can take any value between $0$ and $| A | -1$.

We set
\[
T = \bigl\{ u_1 + u_2 + u_3 \colon u_1 \neq u_2, \, u_3 \notin \{ ku_1,
ku_2\}, \text{ and } u_i \in A_i \bigr\}
\]
and notice that this definition guarantees that $T \subset S$.

Let $H$ be the polynomial of degree $3$
\[
H (X_1, X_2, X_3) = (X_1 - X_2) (X_3 - kX_1)(X_3 - kX_2).
\]
We put
\[
m = c_1+c_2+c_3 - \deg( H) = c_1+c_2+c_3 - 3.
\]
Since $T$ can be rewritten as
\[
T = \{ u_1 + u_2 + u_3 \text{ with } u_i \in A_i,  H (u_1,u_2, u_3)
\neq 0  \},
\]
the polynomial method (see the main theorem from \cite{ANR}) gives
\[
|T| \geq m+1
\]
as soon as we can prove that the coefficient $\chi$ of $X_1^{c_1} X_2^{c_2} X_3^{c_3}$ in
the polynomial
\[
(X_1+X_2+X_3)^m H(X_1, X_2, X_3)
\]
is different from $0$. We are therefore reduced to computing this coefficient.

We compute first that
\[
H(X_1, X_2, X_3) = k^2 X_1X_2 (X_1-X_2) + kX_3(X_2^2 - X_1^2)+X_3^2 (X_1-X_2),
\]
and then see that $\chi$ is the sum of six terms (each being a trinomial coefficient):
\begin{eqnarray*}
\chi & = & k^2 \left( \binomial{\hspace{0.25in}m}{c_1-2 \hspace{0.15in} c_2 -1 \hspace{0.15in} c_3}
- \binomial{\hspace{0.25in}m}{c_1-1 \hspace{0.15in} c_2 -2 \hspace{0.15in} c_3} \right) \\
 && +k \left( \binomial{m\hspace{0.2in}}{c_1 \hspace{0.15in} c_2 -2 \hspace{0.15in} c_3 - 1}
- \binomial{m}{c_1-2 \hspace{0.15in} c_2  \hspace{0.15in} c_3 -1} \right) \\
 && + \left( \binomial{m}{c_1 -1 \hspace{0.15in} c_2 \hspace{0.15in} c_3- 2}
- \binomial{m\hspace{0.2in}}{c_1 \hspace{0.15in} c_2 -1  \hspace{0.15in} c_3 -2} \right),
\end{eqnarray*}
which, after simplification, gives
\begin{equation}
\label{chi}
\chi = \frac{m! (c_1 - c_2)}{c_1! c_2! c_3!}
\left( k^2 c_1 c_2 -k c_3 (c_1 +c_2 -1) + c_3 (c_3 -1) \right).
\end{equation}

We first treat cases (i) and (ii) together and observe that in the case $c_1=c_3=\alpha$ and $c_2=\alpha-1$, where $\alpha$ is a given integer
such that $ 2 \leq \alpha \leq | A |-1 < p$,  we get
\[
\chi=\frac{(3\alpha -4)!}{\alpha!^2\ (\alpha-1)!} \alpha(\alpha-1)(k-1)^2 = \frac{(3\alpha -4)!}{\alpha!  (\alpha-1)! (\alpha-2)!} (k-1)^2.
\]
Since $k \neq 1$, this is non-zero if  $3\alpha-4<p$.

For assertion (i) in the lemma, we thus can choose
$\alpha=|A|-1$ in the preceding computation. This leads to the value $m=3|A|-7$, implying the claim.
For (ii), we choose $\alpha=|A|-2$, also implying the claim.

It remains to consider assertion (iii).
To simplify notations, we write $c_1 =a$  and $c_2= a-x$ with some integer $1\le x\le a$ (implying $c_1 > c_2$).
Finally, we put
\[
c_3 = p+2 -2a+x ;
\]
this choice guarantees that $m=p-1$, yet of course we need to choose $a$ and $x$ in such a way that $c_3 \in [0,|A|-1]$.

With this notation, $\chi \neq 0$ if and only if
\[
k^2 a (a-x) -k (p+2-2a+x) (2a-x -1) + (p+2-2a+x) (p+1-2a+x) \neq 0.
\]
We observe that for each $a$, there is an $x$ in $\{1,2,3\}$ such that this is non-zero;
to see this consider the left-hand side as a  polynomial in $x$; it has degree two, the leading coefficient is $k+1$. Thus it has at most two roots.

The only thing that remains to be shown is that there is an integer $0\le a\le |A|-1 $ such that
$p+2 -2a+x \in [0,|A|-1]$. Yet, indeed, for $ \lceil (p+5)/3\rceil$ this is the case.
\end{proof}

\subsection{Proof of Theorem \ref{theoBehrend}}

We shall construct a set of integers in the interval $\{ 0,1,\dots,\lfloor n/k \rfloor \}$ with the property that
it does not contain a $k$-barycentric subset. It will imply the result by a Freiman isomorphism
(we refer to \cite{Nath, TV} for background and the terminology).

Let us choose an arbitrary integer $m$. There is a unique integer $d$
such that
\begin{equation}
\label{ddd}
( (k-1)(d-1) +1)^m \leq \lfloor n/k \rfloor \leq ((k-1)(d-1)+k)^m -1.
\end{equation}
Notice that this implies
\begin{equation}
\label{ddd2}
d \geq 1+ \frac{(\lfloor n/k \rfloor +1)^{1/m} -k}{k-1} \geq
\frac{(n/k)^{1/m} -k}{k-1}.
\end{equation}

Now any integer less than or equal to $( (k-1)(d-1) +1)^m -1 $ can be written in a unique
way in the form
\[
a=\sum_{i=0}^{m-1} a_i((k-1)(d-1)+1)^i,
\]
where the digits $a_i$'s are integers subject to  $0\le a_i\le (k-1)(d-1)$.
We denote $\phi(a)=(a_0,a_1,\dots, a_{m-1})$ the vector of digits of $a$ in
the chosen basis. For an integer $a$, we define its norm by
\[
\|\phi(a)\|=\Big(\sum_{i=0}^{m-1}a_i^2\Big)^{1/2}.
\]

Let $\ac$ denote the set of integers $0 \leq a \leq  ( (k-1)(d-1) +1)^m -1$
whose digits are at most $d-1$, that is
\begin{eqnarray*}
\ac & = &\{ a_0+a_1((k-1)(d-1)+1)+\cdots+a_{m-1}((k-1)(d-1)+1)^{m-1}, \\
&& \hspace{3.5cm}\text{ with } 0\le a_i\le d-1 \text{ for all } 0 \leq i \leq m-1 \}.
\end{eqnarray*}

By Freiman isomorphism, one has for all $x_1,\dots,x_{k-1}\in \ac$,
\[
\phi(x_1+\cdots +x_{k-1})=\phi(x_1)+\cdots +\phi(x_{k-1})\in
\{0,1,\dots,(k-1)(d-1)\}.
\]

Now, if $r$ is an integer then the set
\[
\ac_r = \bigl\{ a \in \ac \colon \|\phi(a)\|
=\sqrt{r} \bigr\}
\]
does not have a  $k$-barycentric subset because if $x_1+\cdots +x_{k-1}= (k-1)x_k$ with
each of the $x_i$'s in $\ac$, then one has
\begin{eqnarray*}
\|\phi(x_1)\|+\cdots
+\|\phi(x_{k-1})\| & =& (k-1)\sqrt{r}\\
&=&(k-1)\|\phi(x_k)\|\\
&=& \|(k-1)\phi(x_{k})\|\\
&=& \|\phi((k-1)x_{k})\|\\
&=&\|\phi(
x_1 + \cdots x_{k-1})\|\\
&=&\|\phi(x_1)+\cdots +\phi(x_{k-1})\|,
\end{eqnarray*}
so that all the vectors of digits in the chosen basis are
proportional and, since they have the same norm and non-negative coordinates, equal.
By the uniqueness of the writing in the chosen basis
this implies equality of the $x_i$'s and proves the statement.

Now, by the pigeonhole principle, at least one of the $\ac_r$'s is
big. Indeed the cardinality of $\ac$ is equal to $d^m$.
The values taken by
$\|\phi(a)\|^2$, for $a\in \ac\setminus\{0\}$, are integral and at most $m(d-1)^2$. Therefore there exists some $r$ such that
\[
|\ac_r|\ge\frac{d^m-1}{m(d-1)^2}>\frac{d^{m-2}}{m}.
\]

Thus, corresponding to this value of $r$, we have a set without a
$k$-barycentric subset with at least the following number of elements:
\[
\frac{d^{m-2}}{m} \geq
\frac{1}{m} \left( \frac{(n/k)^{1/m} - k}{k-1} \right)^{m-2}
\]
where we have used \eqref{ddd2}.

\subsection{Proof of Corollary \ref{Behrendkfixed}}

Let $\varepsilon_n$ denote the real number such that $k=n^{\varepsilon_n}$.
We notice that $k \ge 3$ implies
\begin{equation}
\label{epsilon}
\varepsilon_n > \frac{1}{\log n},
\end{equation}
and $\log k =o (\log n)$ yields that
\[
\lim_{n \rightarrow + \infty } \varepsilon_n = 0 .
\]

We define
\[
m_0=\left \lfloor \sqrt{ 2 \frac{\log (n/k)}{\log (k-1)}} \right \rfloor
\]
and observe that
\[
m_0 \geq \left\lfloor \sqrt{ 2 \frac{\log (n/k)}{\log k}} \right\rfloor =
\left\lfloor \sqrt{ 2 \frac{1-\varepsilon_n}{\varepsilon_n}} \right\rfloor
 \sim \sqrt{\frac{2}{\varepsilon_n}} \hspace{.5cm} (n \rightarrow + \infty).
\]
Thus $m_0$ tends to infinity as $n$ tends to infinity. Therefore, for $n$ large enough, we have
\begin{equation}
\label{a1}
m_0 \geq \sqrt{ \frac{\log (n/k)}{\log k}} = \sqrt{\frac{1- \varepsilon_n}{\varepsilon_n}}.
\end{equation}

We now prove that the quantity
\[
x_n = m_0\  k \left( \frac{k}{n}\right)^{1/m_0}
\]
tends to 0 when $n$ tends to infinity. Indeed
\[
m_0 \leq \sqrt{ 4 \frac{\log (n/k)}{\log k}}  \leq \frac{2}{\sqrt{\varepsilon_n}} ,
\]
and we get ($n$ large enough)
\begin{eqnarray*}
\log x_n  &\lesssim & \log \left( \frac{2}{\sqrt{\varepsilon_n}} \right) +\varepsilon_n \log n -
		\frac{1}{m_0} (1 - \varepsilon_n) \log n \\
		&\lesssim & \log \left( \frac{2}{\sqrt{\varepsilon_n}} \right) +\varepsilon_n \log n
		- \frac{\sqrt{\varepsilon_n}}{2} (1 -\varepsilon_n ) \log n \\
		&\lesssim & - \frac12 \log \varepsilon_n +\varepsilon_n \log n - \frac{\sqrt{\varepsilon_n }}{3} \log n.
\end{eqnarray*}	
Since $\varepsilon_n$ tends to $0$, it follows
\[
\log x_n \lesssim  - \frac12 \log \varepsilon_n - \frac{\sqrt{\varepsilon_n }}{4} \log n.
\]
Since, by \eqref{epsilon},
\[
| \log \varepsilon_n | < \log \log n
\]
while
\[
\sqrt{\varepsilon_n} \log n \geq \sqrt{\log n},
\]
this second term is dominant and we obtain that
\[
\log x_n \lesssim - \frac14 \sqrt{\varepsilon_n} \log n \lesssim - \frac14 \sqrt{\log n}
\]
and $x_n$ tends to 0 as $n$ goes to infinity which proves our assertion.

We are now ready to prove that
\[
y_n = \left( 1 -  k \left( \frac{k}{n}\right)^{1/m_0} \right)^{m_0 -2}
\]
tends towards $1$ as $n$ tends to infinity. Indeed
\[
\log y_n = (m_0 -2) \log \left( 1 - k \left( \frac{k}{n}\right)^{1/m_0} \right) = (m_0 -2 ) \log \left( 1 - \frac{x_n}{m_0} \right).
\]
Since $m_0$ is at least 1, $x_n /m_0$ itself goes to 0 as $n$ tends to infinity and we obtain
\[
\log y_n = (m_0 -2 ) \log \left( 1 - \frac{x_n}{m_0} \right) \sim m_0 \left( - \frac{x_n}{m_0} \right) = -x_n
\]
which implies that $y_n$ tends towards $1$ when $n$ goes to infinity.

Applying Theorem \ref{theoBehrend}, in which
we plug $m=m_0$, yields
\begin{eqnarray*}
\BO(k, \Z / n \Z)
& \ge & \frac{1}{m_0} \left( \frac{n}{k^{m_0 +1}}\right)^{1-2/m_0}
\left( 1 - \left( \frac{k^{m_0 +1}}{n}\right)^{1/m_0} \right)^{m_0 -2} \\
& = & \frac{1}{m_0} \left( \frac{n}{k^{m_0 +1}}\right)^{1-2/m_0} y_n \\
& \sim &\frac{1}{m_0} \left( \frac{n}{k^{m_0 +1}}\right)^{1-2/m_0} \\
& \geq & \frac{n}{n^{2/m_0} m_0 k^{m_0 }}.
\end{eqnarray*}
Now,
\begin{eqnarray*}
n^{2/m_0} m_0 k^{m_0 } & \leq &
\exp \left(  \frac{2 \log n}{\sqrt{\log (n/k) / \log k}} + \log 2 + \log \log n +2 \sqrt{\log n \log k} \right) \\
& \leq & \exp \bigl( (4+\varepsilon)  \sqrt{\log k \log n} \bigr),
\end{eqnarray*}
for any $\varepsilon >0$, when $n$ is large enough.

It follows
\[
\BO(k, \Z / n \Z) \geq
n \exp   \bigl( - 5 \sqrt{\log k \log n} \bigr).
\]

\section*{Acknowledgements}

The authors thank J.~Wolf for various valuable information, in particular for bringing the reference \cite{alon} to their attention, and the referee for a careful reading and detailed and helpful comments.

\end{document}